\documentclass[12pt,reqno]{amsart}

\usepackage{amssymb}
\usepackage{amsmath}
\usepackage{amscd}
\usepackage[all]{xy}
\usepackage{enumerate,enumitem}

\newtheorem{thm}{Theorem}[section]

\newtheorem{pro}[thm]{Proposition}
\theoremstyle{definition}
\newtheorem{defi}[thm]{Definition}

\newcommand {\emptycomment}[1]{} 

\newcommand{\delete}[1]{}

\title[Non-abelian extensions and Wells exact sequences of Bol algebras ]
{ Non-abelian extensions and Wells exact sequences of Bol algebras }

\author{Jingzi Zhang}
\address{School of Mathematics and Statistics,\\
Henan Normal University, Xinxiang 453007, P. R. China;\\
 E-mail address:\texttt{{  zhangjingzi@163.com}}}

\author{Tao Zhang}
\address{School of Mathematics and Statistics,\\
Henan Normal University, Xinxiang 453007, P. R. China;\\
 E-mail address:\texttt{{  zhangtao@htu.edu.cn}}}

\begin{document}
\begin{abstract}
The aim of this paper is to explore non-abelian extensions of Bol algebras and to study the extensibility of a pair of automorphisms within these non-abelian extensions. We begin by researching  non-abelian extensions of Bol algebras and categorizing them based on non-abelian cohomology groups. Additionally, we discuss the conditions under which a pair of automorphisms can be extended in the context of non-abelian extensions of Bol algebras and derive the corresponding Wells exact sequences.
\end{abstract}

\maketitle

\tableofcontents

\vspace{-1.2cm}

\allowdisplaybreaks

\section{Introduction}

In \cite{Sab1, Mik2}, Bol algebras were initially introduced in the study of the differential geometry of smooth Bol loops. By generalizing the study of the differential geometry of Lie groups, it was found that the tangent structure of a smooth Bol loop is a Bol algebra, and a correspondence between Bol algebras and local smooth Bol loops was discovered locally. More importantly, it was shown in \cite{Mik2} that Bol algebras are Akivis algebras with additional conditions. Then, together with Akivis algebras and Lie-Yamaguti algebras, Bol algebras form an important category within the field of binary-ternary algebras. Among them, the theory of Lie-Yamaguti algebras in the same category as Bol algebras has been studied and explored by many scholars, with detailed research available in \cite{Ya, Ya2, Zhang, Zhao1, ShengZhao}. Initially, general Lie triple systems were referred to as Lie-Yamaguti algebras in \cite{Ya}. Subsequently, Yamaguti investigated the cohomology groups of general Lie triple systems in \cite{Ya2}. Building on this foundation, Zhang and Li  conducted corresponding research on the deformations and extensions of Lie-Yamaguti algebras in \cite{Zhang}. And various aspects of the general theory of Bol algebras have been explored in \cite{Kuz, Hent, Per, Mik1}. For a broader context, one can also refer to \cite{Sab2, Mik3}.

The extension problem is an important issue in group theory and algebras. The concept of non-abelian extensions was first introduced  by Eilenberg and Maclane \cite{Eil}, focusing on the non-abelian extension theory of abstract groups. With the advancement of algebraic geometry and mathematical physics, the theory has been further explored in the context of non-abelian extensions of various types of algebras, including Lie algebras \cite{Fr,Inas},  Lie-Yamaguti algebras \cite{Sun}, Rota-Baxter Lei algebras \cite{Mis}, Rota-Baxter Leibniz algebras \cite{Guo} and associative conformal algebras \cite{Hou}. However, there are relatively few research results on the extension problem for Bol algebras at present. Recently, The abelian extensions of Bol algebras have been examined in \cite{Issa}, but the non-abelian extensions of Bol algebras remain relatively unexplored. We will fill this gap in the present work.

On the basis of the extension theory, the problem of the inducibility of automorphisms is also interesting and important. Under what conditions can a pair of automorphisms be inducible? This question was first posed by Wells \cite{Wells} in the context of abstract groups and was further explored in \cite{Pas,Jin}. In recent years, significant progress has been made in the study of automorphism groups and Wells exact sequences in the context of algebraic extensions. Many scholars have conducted corresponding research on the automorphism groups and Wells exact sequences for various types of algebraic extensions. Das and Ratheeb examined the inducibility of automorphisms in  Rota-Baxter group extensions \cite{Das2}. Du and Tan constructed the Wells exact sequence for the abelian extensions of Lie coalgebras \cite{Du}. Sun, Li, Goswamia, Mishraa and Mukherjee analyzed the inducibility of automorphisms in the extensions of Lie-Yamaguti  algebras \cite{Sun,Gos}. Hazra and Habib investigated the Wells exact sequences in the extensions of Lie superalgebras \cite{Haz}. Inspired by these results, we explore the extendability of a pair of automorphisms on non-abelian extensions of Bol algebras. Furthermore,we provide the necessary and sufficient conditions for a pair of automorphisms to be extensible and derive the analogous Wells short exact sequences in the context of non-abelian extensions of Bol algebras.

This paper is organized as follows. We begin by reviewing the definition of Bol algebras, their representations, and delve into some basic properties of their cohomology groups. Next, we provide an in-depth analysis of non-abelian extensions and utilizes non-abelian cohomology groups for their classification. Then, we tackle the problem of extending pairs of automorphisms within the context of non-abelian extensions of Bol algebras. Finally, we focus on deriving the Wells short exact sequences for non-abelian extensions of Bol algebras.

Throughout this paper, all vector spaces are assumed to be over an algebraically closed field of characteristic different from 2 and 3.
The space of linear maps from two vector spaces $V$ to $W$ is denoted by $\mathrm{Hom} (V,W)$ and the space  of linear maps from $V$ to  itself is denoted by $\mathfrak{gl}(V)$.


\section{Preliminaries on Bol algebras}

We review the concepts of Bol algebras, their representations, and the associated low dimensional cohomology theory. For the details see \cite{Issa}.

\begin{defi} A Bol algebra is a vector space $B$ with a bilinear map $*: B \otimes  B\longrightarrow  B$ and a trilinear map
$[ \ , \ , \ ]: B\otimes  B\otimes  B\longrightarrow  B$,
satisfying
\begin{equation}\label{eq2.1}x_1*x_2+x_2*x_1=0,~~[x_1,x_2,x_3]+[x_2,x_1,x_3]=0,\end{equation}
\begin{equation}\label{eq2.3}[x_1,x_2,x_3]+
[x_2,x_3,x_1]+[x_3,x_1,x_2]=0,\end{equation}
\begin{equation}\label{eq2.4}[x_1,x_2,y_1*y_2]=[x_1,x_2,y_1]*y_2+y_1*[x_1,x_2,y_2]+[y_1,y_2,x_1*x_2] - y_1y_2*x_1x_2,
\end{equation}
\begin{equation}\label{eq2.5}[x_1,x_2,[y_1,y_2,y_3]]=[[x_1,x_2,y_1],y_2,y_3]+[y_1,[x_1,x_2,y_2],y_3]+[y_1,y_2,[x_1,x_2,y_3]],\end{equation}
for all $x_1,x_2,x_3,y_1,y_2,y_3\in  B$.
 Denote it by $( B,*,[ \  , \ , \ ])$ or simply by $ B $  and $y_1y_2*x_1x_2$ means $(y_1*y_2)*(x_1*x_2)$ .
\end{defi}

\begin{defi}(\cite{Issa}) A representation or a modue of a Bol algebra $ B$ consists of a vector
space $V$ together with a linear map $\mu: B\longrightarrow\mathfrak{gl}(V)$
and bilinear maps $\theta,D: B\wedge  B\longrightarrow\mathfrak{gl}(V)$
satisfying
\begin{equation}\label{eq2.9}D(x_1 , x_2) + \theta (x_1 , x_2) - \theta (x_2 , x_1) = 0,\end{equation}
\begin{equation}\label{eq2.10}[D (x_1 , x_2), \mu (y_1)] = \mu ([x_1, x_2, y_1]) - \theta (y_1, x_1 * x_2)
+ \mu (x_1 * x_2)\mu (y_1),\end{equation}
\begin{equation}\label{eq2.11}\theta (x_1, y_1 * y_2) = \mu (y_1) \theta (x_1, y_2) - \mu (y_2) \theta (x_1, y_1)
- {\big (} D(y_1 , y_2) -  \mu (y_1 * y_2) {\big )} \mu (x_1),\end{equation}
\begin{equation}\label{eq2.12}[D (x_1 , x_2), D (y_1 , y_2)] = D ([x_1 , x_2, y_1], y_2) + D (y_1, [x_1 , x_2, y_2]),\end{equation}
\begin{equation}\label{eq2.13}[D (x_1 , x_2), \theta (y_1, y_2)] = \theta ([x_1 , x_2, y_1], y_2)
+ \theta (y_1, [x_1 , x_2, y_2]),\end{equation}
\begin{equation}\label{eq2.7}\theta (x_1 , [y_1, y_2, y_3]) = \theta (y_2, y_3) \theta (x_1, y_1)
- \theta (y_1, y_3) \theta (x_1, y_2) + D (y_1, y_2)\theta (x_1, y_3),\end{equation}
for all $x_i, y_i\in  B$. Denote the representation of $ B$ by
$(V,\mu,\theta,D)$ or simply by $V$.\end{defi}

\begin{pro}(\cite{Issa})  Let $( B,*_{ B},[ \  , \ , \ ]_{ B})$ be
a Bol algebra and $V$ be a vector space. Assume that $\mu: B\longrightarrow\mathfrak{gl}(V)$
is a linear map and  $\theta,D: B\wedge  B\longrightarrow\mathfrak{gl}(V)$ are bilinear maps.
Then $(V,\mu,\theta,D)$ is a representation of $ B$ if and only
if $( B\oplus V,*,[  \  , \ , \ ])$ is a Bol algebra, where
\begin{equation*}[x+u,y+v,z+w]=[x,y,z]_{ B}+\theta(y,z)u-\theta(x,z)v+D(x,y)w,\end{equation*}
and
\begin{equation*}(x+u)*(y+v)=x*_{ B}y+\mu(x)v-\mu(y)u,\end{equation*}
for all $x,y,z\in B,u,v,w\in V$.
The Bol algebra $( B\oplus V,*,[  \  , \ , \ ])$ is called
the semidirect product Bol algebra. Denote it simply by $ B\ltimes V$.
\end{pro}

Because of applications in a study of deformations and abelian extensions of Bol algebras in \cite{Issa}, we
restrict ourself to a construction of $(2,3)$-cohomology groups for this class of algebras. Let $C^2 (B,V)$ be the space of bilinear maps $\nu : B \times B \rightarrow V$ such that  $\nu (x_1, x_2) = - \nu (x_2, x_1)$ and $C^3 (B,V)$ be the space of trilinear maps
$\omega : B \times B \times B \rightarrow V$ such that $\omega (x_1, x_2, x_3) = - \omega (x_2, x_1, x_3)$ for all $x_i \in B$.
\begin{defi} Let $B$ be a Bol algebra and $(\mu, D, \theta)$ be a representation of $ B$. The pair $(\nu, \omega) \in C^2 (B,V) \times C^3 (B,V)$ is called a
$(2,3)$-cocycle , for all $x_i, y_i \in B$, we have:
\begin{equation} \omega (x_1, x_2, x_3) + \omega (x_2, x_3, x_1)+ \omega (x_3, x_1, x_2)= 0,\end{equation}
\begin{align} &\omega (x_1, x_2, y_1 * y_2) + D(x_1, x_2) \nu (y_1, y_2)\\
&= \omega (y_1, y_2, x_1 * x_2) + D(y_1, y_2) \nu (x_1, x_2)+ \nu ([x_1, x_2, y_1], y_2) \nonumber\\
&+ \nu (y_1, [x_1, x_2, y_2])+ \mu (y_1) \omega (x_1, x_2, y_2) - \mu (y_2) \omega (x_1, x_2, y_1)\nonumber\\
&+ \nu (x_1 * x_2) \nu (y_1, y_2) - \mu (y_1 * y_2) \nu (x_1, x_2)- \nu (y_1 * y_2, x_1 * x_2),\nonumber\end{align}
\begin{align}  &\omega (x_1, x_2, [y_1, y_2, y_3]) + D(x_1, x_2) \omega (y_1, y_2, y_3)\\
&= \omega ([x_1, x_2, y_1], y_2, y_3) + \omega (y_1, [x_1, x_2, y_2], y_3)+ \omega (y_1, y_2, [x_1, x_2, y_3])\nonumber\\
&+ D(y_1, y_2) \omega (x_1, x_2, y_3) + \theta (y_2, y_3) \omega (x_1, x_2, y_1)- \theta (y_1, y_3) \omega (x_1, x_2, y_2).\nonumber \end{align}
\end{defi}
The vector space generated by all (2,3)-cocycles is represented as $Z^2(B,V)\times Z^3(B,V)$.

\begin{defi} The $f$ be a linear mapping of $B$ into a representation space $V$. Then $f$ is called a pseudoderivation
with companion $\chi \in V$ of $B$ into $V$ with respect to the representation  $(\mu, D, \theta)$, for all $x_i \in B$,
\begin{align}f(x_1 * x_2) = \mu (x_1) f(x_2) - \mu (x_2) f(x_1) +  {\Delta}_{D, \mu} (x_1 , x_2)(\chi),
\end{align}
\begin{align}
f([x_1, x_2, x_3]) = \theta (x_2, x_3) f(x_1) - \theta (x_1, x_3) f(x_2)
+ D(x_1, x_2) f(x_3).
\end{align}
\end{defi}

\begin{defi}Let $B$ be a Bol algebra, and $(\mu, D, \theta)$ be a representation of $B$.
Then the pair $(\nu,\omega)$ where $\nu:B\otimes B\longrightarrow V$, $\omega:B \otimes B \otimes B \longrightarrow  V$ is called a
$(2,3)$-coboundary if there exists a pair $(f, \chi)$, where $f:B \rightarrow V$ is a linear
map and $\chi \in V$ such that
\begin{align}
 \nu (x_1, x_2) = \mu (x_1) f(x_2) - \mu (x_2) f(x_1)
+  (D(x_1,x_2)-\mu(x_1 * x_2))(\chi) - f(x_1 * x_2),
\end{align}
\begin{align}
  \omega (x_1, x_2, x_3) = \theta (x_2, x_3) f(x_1) - \theta (x_1, x_3) f(x_2)
+ D(x_1, x_2) f(x_3) - f([x_1, x_2, x_3]),
\end{align}
for all $x_i \in B$. The space of $(2,3)$-coboundaries is denoted by $B^2 (B,V) \times B^3 (B,V)$.
\end{defi}

It is proved in \cite{Issa} that the space of $(2,3)$-coboundaries is contained in space of $(2,3)$-cocycles.
Thus we have

\begin{defi}(\cite{Issa}) The $(2,3)$-cohomology group of a Bol algebra $B$ with coefficients
in the $B$-module $V$ is the quotient space
$$H^{(2,3)} (B,V):= (Z^2 (B,V) \times Z^3 (B,V)) / (B^2 (B,V) \times B^3 (B,V)).$$ \end{defi}
\section{Non-abelian extensions of Bol algebras}

In this section, we delve into the study of non-abelian extensions and non-abelian (2,3)-cocycles within the context of Bol algebras.

\begin{defi}
Let $B$ and $V$ be two Bol algebras, a non-abelian extension of $B$ by $V$ is described by a Bol algebra \(\hat{B}\), which fits into the following short exact sequence of Bol algebras:
$$\mathcal{E}:0\longrightarrow V\stackrel{i}{\longrightarrow} \hat B\stackrel{p}{\longrightarrow} B\longrightarrow0.$$
For ease of notation, we represent this extension as  \(\mathcal{E}\).
A section of \(p\) is a linear map \(s:B\to\hat{B}\) such that \(ps=\text{id}_B\).
\end{defi}

\begin{defi}
Consider two non-abelian extensions of $B$ by $V$, labeled as \(\hat{B}_1\) and \(\hat{B}_2\). These extensions are deemed equivalent if a homomorphism \(f:\hat{B}_1\to\hat{B}_2\) exists, which ensures the commutativity of the following diagram:
 \begin{equation}\label{Ene1} \xymatrix{
  0 \ar[r] & V\ar@{=}[d] \ar[r]^{i_1} & \hat B_1\ar[d]_{f} \ar[r]^{p_1} & B \ar@{=}[d] \ar[r] & 0\\
 0 \ar[r] & V \ar[r]^{i_2} & \hat B_2 \ar[r]^{p_2} & B  \ar[r] & 0
 .}\end{equation}
\end{defi}
Denote by $\mathcal{E}_{nab}(B,V)$ the set of all non-abelian extensions of $B$ by $V$.

Subsequently, we present the notion of a non-abelian cohomology group and illustrate how the categorization of non-abelian extensions is governed by these cohomology groups.

\begin{defi} Let $B$ and $V$ be two Bol algebras.
A non-abelian (2,3)-cocycle on $B$ with values in  $V$ is a septuple $(\nu,\omega,\mu,\theta,D)$
of maps such that $\omega:B \otimes B \otimes B \longrightarrow  V$ is trilinear, $\nu:B\otimes B\longrightarrow V,
~\theta,D:B\wedge B\longrightarrow \mathfrak{gl}(V)$ are bilinear
and $\mu:B\longrightarrow \mathfrak{gl}(V)$ is linear,
and the following these identities are satisfied for all
$x_i,y_i, x, y, z\in B~ (i=1,2,3), a,b,c\in V$,
\begin{equation}\label{B00}\nu(x,y)+\nu(y,x)
=0,~~\omega(x,y,z)+\omega(y,x,z)=0,
\end{equation}
\begin{align}
\label{B12}\omega(x,y,z)+\omega(y,z,x)+\omega(z,x,y)=0,
\end{align}
\begin{equation}
\label{B01}D(x,y)a+D(y,x)a=0
,\end{equation}
\begin{equation}\label{B13}D(x,y)a-\theta(y,x)a+\theta(x,y)a
=0,\end{equation}
\begin{align}\label{B22}&D(x,y)\nu(z,w)+\omega(x,y,z*_B w)\\
=&\nu([x,y,z]_B,w)-\mu(w)\omega(x,y,z)+\mu(z)\omega(x,y,w)+\nu(z,[x,y,w]_B)\nonumber\\
&+\omega(z,w,x*_B y)+D(z,w)\nu(x,y)-\nu(z*_B w,x*_B y)\nonumber\\
&-\mu(x*_B y)\nu(x,y)+\mu(x*_B y)\nu(z,w),\nonumber
\end{align}
\begin{align}\label{B23}
&D(x,y)\mu(z)a+\theta(z,x*_B y)a\\
=&\mu([x,y,z]_B)a+\omega(x,y,z)*_V a+\mu(z)D(x,y)a+\mu(x*_B y)\mu(z)a,\nonumber
\end{align}
\begin{align}
\label{B24}&D(x,y)(a*_V b)+\mu(x*_B y)(a*_V b)\\
=&(D(x,y)a)*_V b+a*_V (D(x,y)b)+[a,b,\nu(x,y)]\nonumber\end{align}
\begin{align}
\label{B25}[a,b,\nu(x,y)]_V+D(x,y)(a*_ V b)+\mu(x*_B y)(a*_ V b)=0,
\end{align}
\begin{align}\label{B31}\theta(x,[y,z,w]_B)a
 =\theta(z,w)\theta(x,y)a-\theta(y,w)\theta(x,z)a+D(y,z)\theta(x,w)a,
 \end{align}
\begin{align}\label{B32}\theta(x,y*_B z)a
=\mu(y)\theta(x,z)a-\mu(z)\theta(x,y)a-D(y,z)\mu(u)a+\mu(y*_B z)\mu(x)a,
\end{align}
\begin{align}
\label{B2}D(x,y)\theta(z,w)a-\theta(z,w)D(x,y)a=\theta([x,y,z]_B,w)a+\theta(z,[x,y,w]_B)a,
 \end{align}
 \begin{align}
 \label{B3}D(x,y)D(z,w)a-D(z,w)D(x,y)a=D([x,y,z]_B,w)a+D(z,[x,y,w ]_B)a,
\end{align}
\begin{equation}
\label{B4}D(x,y)([a,b,c]_V)=[D(x,y)a,b,c]_V+[a,D(x,y)b,c]_V+[a,b,D(x,y)c]_V.
 \end{equation}
 \begin{align}\label{B1}&D(x_1,x_2)\omega(y_1,y_2,y_3)+\omega(x_1,x_2,[y_1,y_2,y_3]_B)\\
=&\omega([x_1,x_2,y_1]_B,y_2,y_3)+\theta(y_2,y_3)\omega(x_1,x_2,y_1)+\omega(y_1,[x_1,x_2,y_2]_B,y_3)\nonumber\\
&-\theta(y_1,y_3)\omega(x_1,x_2,y_2)+\omega(y_1,y_2,[x_1,x_2,y_3]_B)+D(y_1,y_2)\omega(x_1,x_2,y_3),\nonumber
 \end{align}
\end{defi}

\begin{defi} Suppose we have two non-abelian (2,3)-cocycles on $B$ with values in $V$, denoted as \((\nu_1,\omega_1,\mu_1,\theta_1,D_1)\) and \((\nu_2,\omega_2,\mu_2,\theta_2,D_2)\). These cocycles are considered equivalent if there exists a linear map \(\varphi:B\to V\) that satisfies the following equalities for all \(x,y,z\in B\) and \(a,b\in V\):
\begin{align}\label{E1}&\omega_1(x,y,z)-\omega_2(x,y,z)\\
=&\theta_{2}(x,z)\varphi(y)-D_{2}(x,y)\varphi(z)-\theta_{2}(y,z)\varphi(x)\nonumber\\
&-[\varphi(x),\varphi(y),\varphi(z)]_V+\varphi[x,y,z]_B,\nonumber
 \end{align}
 \begin{align}\label{E2}
 \nu_1(x,y)-\nu_2(x,y)=&\varphi(x) *_V \varphi(y)+\varphi (x*_B y)\\
 &-\mu_{2}(x)\varphi(y)+\mu_{2}(y)\varphi(x),\nonumber
 \end{align}
\begin{equation}\label{E3}\mu_1(x)a-\mu_2(x)a=a*_V \varphi(x),\end{equation}
   \begin{equation}\label{E4}\theta_1(x,y)a-\theta_2(x,y)a=[a,\varphi(x),\varphi(y)]_V,
 \end{equation}
 \begin{equation}\label{E5}D_{1}(x,y)a-D_{2}(x,y)a=[\varphi(x),\varphi(y),a]_V,
 \end{equation}
 \end{defi}
To streamline notation, we also refer to a non-abelian (2,3)-cocycle \((\nu,\omega,\mu,\theta,D)\) simply as \((\nu,\omega)\). The equivalence class of a non-abelian (2,3)-cocycle \((\nu,\omega,\mu,\theta,D)\) is denoted by \([(\nu,\omega)]\). Additionally, the collection of all equivalence classes of non-abelian (2,3)-cocycles is represented by \(H^{(2,3)}_{\text{nab}}(B,V)\).

Using the above notations, we define multilinear maps $*_{\nu}$ and $[  \ , \ , \ ]_{\omega}$ on $ B \oplus V $ by
\begin{align}\label{NLts0}(x+a)*_{\nu}(y+b)&=x*_B y+\nu(x,y)+\mu(x)b-\mu(y)a+a*_V b
,\end{align}
\begin{align}\label{NLts}[x+a,y+b,z+c]_{\omega}=&[x,y,z]_B+\omega(x,y,z)+D(x,y)c+\theta(y,z)a
\\&-\theta(x,z)b+[a,b,c]_V,\nonumber\end{align}
for all $x,y,z\in B$ and $a,b,c\in V$.

\begin{pro} \label{BOL}Given the established notations above,
the structure $(B\oplus V,*_{\nu},[  \  , \ , \
]_{\omega})$ forms a Bol algebra if and only if the collection $(\nu,\omega,\mu,\theta,D)$ constitutes a non-abelian (2,3)-cocycle.
For simplicity, we refer to this Bol algebra as $B\oplus_{(\nu,\omega)}V$.
\end{pro}

The proof of the above Proposition \ref{BOL} is by direct computations, so we omit the details.

Let
 $\mathcal{E}:0\longrightarrow V\stackrel{i}{\longrightarrow} \hat B \stackrel{p}{\longrightarrow} B \longrightarrow0$
be a non-abelian extension of  $B$ by
$V$ with a section $s$ of $p$.
Define $\nu_{s}:B\otimes B\longrightarrow V,~\omega_{s}:B \otimes B  \otimes B \longrightarrow V,~\mu_{s}:B
\longrightarrow \mathfrak{gl}(V),~\theta_{s},D_{s}:B \wedge B \longrightarrow \mathfrak{gl}(V)$
 respectively by
 \begin{equation}\label{C0}\nu_{s}(x,y)=s(x)*_{\hat{B}}s(y)-s(x*_{B}y),\end{equation}
\begin{equation}\label{C1}\omega_{s}(x,y,z)=[s(x),s(y),s(z)]_{\hat{B}}-s[x,y,z]_{B},\end{equation}
\begin{equation}\label{C2}\theta_{s}(x,y)a=[a,s(x),s(y)]_{\hat{B}},\end{equation}
\begin{equation}\label{C3}D_{s}(x,y)a=[s(x),s(y),a]_{\hat{B}},\end{equation}
\begin{equation}\label{C4}\mu_{s}(x)a=s(x)*_{\hat{B}}a\end{equation}
for all $x,y,z\in B,a,b\in V$.

Through straightforward calculations, we arrive at the following results:

\begin{pro} \label{CB}
With the above notations, the tuple $(\nu_{s},\omega_{s},\mu_{s},\theta_{s},D_{s})$ forms a non-abelian (2,3)-cocycle on $B$ with values in $V$.
Moreover, two equivalent non-abelian extension give equivalent 2-cocycles.
\end{pro}

\begin{proof}
Since $p$ is a Bol algebra homomorphism, we have
\begin{eqnarray*}
p\omega(x, y, z)&=&p([s(x),s(y),s(z)]_{\hat{B}})-p(s[x,y,z]_{B})
\\&=&([ps(x),ps(y),ps(z)]_{\hat{B}})-(ps[x,y,z]_{B})=[x, y, z]-[x, y, z]=0.
\end{eqnarray*}
Thus $\omega(x, y, z)\in Ker p=V$.
Similarly, one show $\nu(x, y)\in Ker p=V$.

By the equality
\begin{eqnarray*}
&&[s(x_1), s(x_2), [s(y_1), s(y_2), s(y_3)]_{\hat{B}}]_{\hat{B}}\\
&=&[[s(x_1), s(x_2), s(y_1)]_{\hat{B}}, s(y_2), s(y_3)]_{\hat{B}} + [ s(y_1), [s(x_1), s(x_2), s(y_2)]_{\hat{B}}, s(y_3)]_{\hat{B}} \\
&&+ [s(y_1), s(y_2), [s(x_1), s(x_2), s(y_3)]_{\hat{B}}]_{\hat{B}}.
\end{eqnarray*}
Expanding the left side of the equation yields
\begin{eqnarray*}
&&[s(x_1), s(x_2), [s(y_1), s(y_2), s(y_3)]_{\hat{B}}]_{\hat{B}}\\
&=&[s(x_1), s(x_2), s[y_1, y_2, y_3]_B+\omega(y_1, y_2, y_3)]_{\hat{B}}\\
&=&s([x_1, x_2, [y_1, y_2, y_3]_B]_B)+D(x_1,x_2)\omega(y_1, y_2, y_3)+\omega(x_1, x_2, [y_1, y_2, y_3]_B).
\end{eqnarray*}
Expanding the right side of the equation yields
\begin{eqnarray*}
&&[[s(x_1), s(x_2), s(y_1)]_{\hat{B}}, s(y_2), s(y_3)]_{\hat{B}} + [ s(y_1), [s(x_1), s(x_2), s(y_2)]_{\hat{B}}, s(y_3)]_{\hat{B}}\\
&&+ [s(y_1), s(y_2), [s(x_1), s(x_2), s(y_3)]_{\hat{B}}]_{\hat{B}}\\
&=&[s[x_1, x_2, y_1]_B+\omega(x_1, x_2, y_1), s(y_2), s(y_3)]_{\hat{B}}
+[s(y_1), s[x_1, x_2, y_2]_B+\omega(x_1, x_2, y_2), s(y_3)]_{\hat{B}}\\
&&+[s(y_1), s(y_2), s[x_1, x_2, y_3]_B+\omega(x_1, x_2, y_3)]_{\hat{B}}\\
&=&s([[x_1, x_2, y_1]_B, y_2, y_3]_B)+\theta(y_2, y_3)\omega(x_1, x_2, y_1)+\omega([x_1, x_2, y_1]_B, y_2, y_3)\\
&&+s([y_1,s[x_1, x_2, y_2]_B, y_3]_B)-\theta(y_1, y_3)\omega(x_1, x_2, y_2)+\omega(y_1,s[x_1, x_2, y_2]_B, y_3)\\
&&+s([y_1, y_2, s[x_1, x_2, y_3]_B]_B)+D(y_1, y_2)\omega(x_1, x_2, y_3)+\omega(y_1, y_2, s[x_1, x_2, y_3]_B).
\end{eqnarray*}
Thus, we have
\begin{eqnarray*}
&&D(x_1,x_2)\omega(y_1, y_2, y_3)+\omega(x_1, x_2, [y_1, y_2, y_3]_B)\\&=&\theta(y_2, y_3)\omega(x_1, x_2, y_1)+\omega([x_1, x_2, y_1]_B, y_2, y_3)-\theta(y_1, y_3)\omega(x_1, x_2, y_2)\\&&+\omega(y_1,s[x_1, x_2, y_2]_B, y_3)+D(y_1, y_2)\omega(x_1, x_2, y_3)+\omega(y_1, y_2, s[x_1, x_2, y_3]_B).
\end{eqnarray*}
Therefore we obtain that the non-abelian (2, 3)-cocycle condition \eqref{B1} holds.  Similarly, one show that the conditions \eqref{B00}--\eqref{B4} hold too. Thus $(\nu,\omega,\mu,\theta,D)$ is a non-abelian (2, 3)-cocycle.

Now, we show that $(\nu,\omega,\mu,\theta,D)$ does not depend on the choice of the section $s$. Let $s$ and $s'$ be two section and denote by $\varphi(x)=s(x)-s'(x)$. Then we have $p(\varphi(x))=p(s(x))-p(s'(x))=0$, thus $\varphi(B)\subseteq  Ker p  \cong   V$, that is $\varphi \in \mathrm{Hom}(B, V)$.

We compute
\begin{eqnarray*}
&&\omega(x, y, z)-\omega'(x, y, z)\\
&=&([s(x),s(y),s(z)]_{\hat{B}}-s[x,y,z]_{B})-([s'(x),s'(y),s'(z)]_{\hat{B}}-s'[x,y,z]_{B})\\
&=&-[s'(x),\varphi(y),s'(z)]_{\hat{B}}-[s'(x),s'(y),\varphi(z)]_{\hat{B}}-[\varphi(x),s'(y),s'(z)]_{\hat{B}}\\
&&-[\varphi(x),\varphi(y),\varphi(z)]_{\hat{B}}+\varphi[x,y,z]_{B}\\
&=&\theta'(x,z)\varphi(y)-D'(x,y)\varphi(z)-\theta'(y,z)\varphi(x)-[\varphi(x),\varphi(y),\varphi(z)]_{\hat{B}}+\varphi[x,y,z]_{B}.
\end{eqnarray*} which yields that
 Eq.~(\ref{E1}) holds.
\begin{eqnarray*}
&&\nu(x, y)-\nu'(x, y)\\
&=&(s(x)*_{\hat{B}}s(y)-s(x*_{B}y))-(s'(x)*_{\hat{B}}s'(y)-s'(x*_{B}y))\\
&=&(s(x)*_{\hat{B}}s(y)-s'(x)*_{\hat{B}}s'(y))-(s(x*_{B}y)-s'(x*_{B}y))\\
&=&\varphi(x) *_{\hat{B}} \varphi(y)+\varphi (x*_B y) -\mu'(x)\varphi(y)+\mu'(y)\varphi(x).
\end{eqnarray*} which yields that
 Eq.~(\ref{E2}) holds.
\begin{eqnarray*}
&&\theta(x, y)a-\theta'(x, y)a\\
&=&[a,s(x),s(y)]_{\hat{B}}-[a,s'(x),s'(y)]_{\hat{B}}\\
&=&[a,\varphi(x),\varphi(y)]_{\hat{B}}.
\end{eqnarray*} which yields that
 Eq.~(\ref{E4}) holds.
\begin{eqnarray*}
&&D(x, y)a-D'(x, y)a\\
&=&[s(x),s(y),a]_{\hat{B}}-[s'(x),s'(y),a]_{\hat{B}}\\
&=&[\varphi(x),\varphi(y), a]_{\hat{B}}.
\end{eqnarray*} which yields that
 Eq.~(\ref{E5}) holds. This finishes the proof.
\end{proof}

Based on  Proposition \ref{BOL} and Proposition \ref{CB}, for a non-abelian extension $\mathcal{E}:0\longrightarrow V\stackrel{i}{\longrightarrow} \hat{B}\stackrel{p}{\longrightarrow}B\longrightarrow0$ with a section $s$, we obtain a non-abelian (2,3)-cocycle and the corresponding Bol algebra $B\oplus_{(\nu_s,\omega_s)} V$. Consequently, $\mathcal{E}_{(\nu_s,\omega_s)}:0\longrightarrow V\stackrel{i}{\longrightarrow} B\oplus_{(\nu_s,\omega_s)} V\stackrel{\pi}{\longrightarrow}B\longrightarrow0$ forms a non-abelian extension.
Since any element $\hat{w}\in \hat{B}$ can be expressed as $\hat{w}=a+s(x)$, where $a\in V,x\in B$, we define a linear map \(f:\hat{B}\rightarrow B\oplus_{(\nu_s,\omega_s)}V\) such that \(f(\hat{w})=f(a+s(x))=a+x\). It is easy to verify that $f$ is an isomorphism of Bol algebras and satisfies the following commutative diagram:
 \begin{equation*} \xymatrix{
  \mathcal{E}:0 \ar[r] & V\ar@{=}[d] \ar[r]^-{i} & \hat{B}\ar[d]_-{f} \ar[r]^-{p} & B \ar@{=}[d] \ar[r] & 0\\
 \mathcal{E}_{(\nu_s,\omega_s)}:0 \ar[r] & V \ar[r]^-{i} & B\oplus_{(\nu_s,\omega_s)} V \ar[r]^-{\pi} & B  \ar[r] & 0,}\end{equation*}

Next, we delve into the connection between non-abelian (2,3)-cocycles and the extensions they induce.

\begin{pro}
Consider two Bol algebras $B$ and $V$. The equivalence classes of non-abelian extensions of $B$ by $V$ can be classified by using the non-abelian cohomology group. Specifically, there is a bijiection map between the set of equivalence classes of non-abelian extensions \(\mathcal{E}_{nab}(B,V)\) and the non-abelian cohomology group \(H^{(2,3)}_{nab}(B,V)\).
 \end{pro}

 \begin{proof}
 Define the map
 \begin{equation*}
 \Theta:\mathcal{E}_{nab}(B,V)\rightarrow H_{nab}^{(2,3)}(B,V),\quad  \mathcal{E}\mapsto (\nu_s,\omega_s).
 \end{equation*}
 Initially, we establish that $\Theta$ is well-defined. Suppose $\mathcal{E}_1$ and $\mathcal{E}_2$ are two non-abelian extensions of $B$ by $V$ that are equivalent via a map $f$, ensuring the commutative diagram (\ref{Ene1}) holds. Given a section \(s_1:B\to\hat{B}_1\) for \(p_1\), we have \(p_2 f s_1=p_1 s_1=\text{id}_B\), which implies \(s_2=f s_1\) is a section for \(p_2\). Now, consider two non-abelian (2,3)-cocycles \((\nu_1,\omega_1,\mu_1,\theta_1,D_1)\) and \((\nu_2,\omega_2,\mu_2,\theta_2,D_2)\), induced by the sections \(s_1\) and \(s_2\) respectively. The mapping $\Theta$ then correlates the equivalence class of the non-abelian extension with the corresponding class of non-abelian (2,3)-cocycles. Then we find that,
\begin{eqnarray*}D_{1}(x,y)a&=&f(D_1(x,y)a)=f([s_1(x),s_1(y),a]_{\hat{B}_1})
\\&=&[fs_1(x),fs_1(y),f(a)]_{\hat{B}_2}\\&=&[s_2(x),s_2(y),a]_{\hat{B}_2}\\&=&D_{2}(x,y)a .\end{eqnarray*} By the same token, we have
\begin{equation*}\theta_1(x,y)a=\theta_2(x,y)a,\nu_1(x,y)=\nu_2(x,y),\omega_1(x,y,z)=\omega_2(x,y,z),\end{equation*}
Thus, we can proof  that $\Theta$ is well-defined.

 Next, we verify that $\Theta$ is injective. Indeed, suppose that
 $\Theta([\mathcal{E}_1])=[(\nu_1,\omega_1)]$ and $\Theta([\mathcal{E}_2])=[(\nu_2,\omega_2)]$. If
the equivalent classes $[(\nu_1,\omega_1)]=[(\nu_2,\omega_2)]$, we obtain that the non-abelian (2,3)-cocycles
 $(\nu_1,\omega_1,\mu_1,\theta_1,D_{1})$ and
$(\nu_2,\omega_2,\mu_2,\theta_2,D_{2})$ are equivalent
via the linear map $\varphi:B\longrightarrow
V$, satisfying Eqs.~~(\ref{E1})-(\ref{E5}). Define a linear map
$f:B\oplus_{(\nu_1,\omega_1)} V\longrightarrow B\oplus_{(\nu_2,\omega_2)}
V$ by
\begin{equation*}f(x+a)=x-\varphi(x)+a,~~\forall~x\in B,a\in V.\end{equation*}
 According to Eq.~~(\ref{NLts0}), for all $x,y,z\in B,a,b,c\in V$, we get
\begin{eqnarray*}&&f((x+a)*_{\nu_1}(y+b))
\\&=&f(x*_{B}y+\nu(x,y)+\mu_1(x)b-\mu_1(y)a+a*_{V} b)
\\&=&x*_{B}y-\varphi(x*_{B}y)+\nu_1(x,y)+\mu_1(x)b-\mu_1(y)a+a*_{V}b,\end{eqnarray*}
and
\begin{eqnarray*}&&f(x+a)*_{\nu_2}f(y+b)
\\&=&[(x-\varphi(x)+a)*_{\nu_2}(y-\varphi(y)+b)]
\\&=&x*_{B}y+\nu_2(x,y)+\mu_2(x)(b-\varphi(y))-\mu_2(y)(a-\varphi(x))+(a-\varphi(x))*_{V}(b-\varphi(y))
\\&=&x*_{B}y+\nu_2(x,y)+\mu_2(x)b-\mu_2(x)\varphi(y)-\mu_2(y)a-\mu_2(y)\varphi(x)+a*_{V}b-a*_{V}\varphi(y)
\\&&-\varphi(x)*_{V}b+\varphi(x)*_{V}\varphi(y)
.\end{eqnarray*}
In view of Eqs.~(\ref{E1})-(\ref{E5}), we have
$f((x+a)*_{\nu_1}(y+b))=f(x+a)*_{\nu_2}f(y+b)$. Similarly, $f([x+a,y+b,z+c]_{\omega_1})=[f(x+a),f(y+b),f(z+c)]_{\omega_2}$.
Hence, $f$ is a homomorphism of Bol algebras. Clearly, the
following commutative diagram holds:
\begin{equation}
\xymatrix{
 \mathcal{E}_{(\nu_1,\omega_1)}: 0 \ar[r] & V\ar@{=}[d] \ar[r]^-{i} & B\oplus_{(\nu_1,\omega_1)} V \ar[d]_-{f} \ar[r]^-{\pi} & B \ar@{=}[d] \ar[r] & 0\\
\mathcal{E}_{(\nu_2,\omega_2)}: 0 \ar[r] & V \ar[r]^-{i} & B\oplus_{(\nu_2,\omega_2)} V \ar[r]^-{\pi} & B  \ar[r] & 0
 .}\end{equation}
Thus $\mathcal{E}_{(\nu_1,\omega_1)}$ and $\mathcal{E}_{(\nu_2,\omega_2)}$ are equivalent
non-abelian extensions of $B$ by $V$,
which means that $[\mathcal{E}_{(\nu_1,\omega_1)}]=[\mathcal{E}_{(\nu_2,\omega_2)}]$. Thus, $\Theta$ is injective.

 Finally, we claim that $\Theta$ is surjective.
 For any equivalent class of non-abelian (2,3)-cocycles $[(\nu,\omega)]$, by Proposition \ref{BOL}, there is
 a non-abelian extension of $B$ by $V$:
   \begin{equation*}\mathcal{E}_{(\nu,\omega)}:0\longrightarrow V\stackrel{i}{\longrightarrow} B\oplus_{(\nu,\omega)} V\stackrel{\pi}{\longrightarrow} B\longrightarrow0.\end{equation*}
   Thus we get, $\Theta([\mathcal{E}_{(\nu,\omega)}])=[(\nu,\omega)]$, which follows that $\Theta$ is surjective.
    Therefore, $\Theta$ is bijective. This finishes the proof.

 \end{proof}


\section{Extensibility of a pair of Bol algebra automorphisms}
In this section, we delve into the conditions under which pairs of automorphisms of a Bol algebra can be extended, and we provide a characterization of these pairs through equivalent conditions.

Let $\mathcal{E}:0\longrightarrow V\stackrel{i}{\longrightarrow} \hat{B}\stackrel{p}{\longrightarrow}B\longrightarrow0$
  be a non-abelian extension of $B$ by $V$ with a section $s$ of $p$.
Denote  $\mathrm{Aut}_{V}(\hat{B})=\{\gamma\in \mathrm{Aut} (\hat{B})\mid \gamma(V)=V\}.$

\begin{defi}

Consider a non-abelian extension $$\mathcal{E}:0\longrightarrow V\stackrel{i}{\longrightarrow} \hat{B}\stackrel{p}{\longrightarrow}B\longrightarrow0.$$A pair of automorphisms  $(\alpha,\beta)\in \mathrm{Aut} (B)\times \mathrm{Aut} (V)$ is termed inducible if there exists an automorphism $\gamma\in \mathrm{Aut}_{V}
(\hat{B})$ such that $i\beta=\gamma i,~p\gamma=\alpha p$. In other words, the following commutative diagram is satisfied:
\[\xymatrix@C=20pt@R=20pt{0\ar[r]&V\ar[d]_\beta\ar[r]^i&\hat{B}\ar[d]_\gamma\ar[r]^p&B
\ar[d]_\alpha\ar[r]&0\\0\ar[r]&V\ar[r]^i&\hat{B}\ar[r]^p&B\ar[r]&0.}\]
\end{defi}
It is pertinent to inquire under what conditions a pair of automorphisms $(\alpha,\beta)\in \mathrm{Aut} (B)\times \mathrm{Aut} (V)$ can be deemed inducible. We will explore this question in the subsequent discussion.

\begin{thm} \label{EC}
Consider a non-abelian extension of $B$ by
$V$ given by $0\longrightarrow V\stackrel{i}{\longrightarrow}
\hat{B}\stackrel{p}{\longrightarrow}B\longrightarrow0$, equipped with a section $s$ of $p$. Let $(\nu,\omega,\mu,\theta,D)$ denote the non-abelian (2,3)-cocycle induced by $s$. A pair of automorphisms $(\alpha,\beta)\in \mathrm{Aut}(B)\times \mathrm{Aut}(V)$ can be induced if and only if there exists a linear map $\varphi:B\longrightarrow V$ that fulfills the following conditions:
\begin{align}\label{Iam1}
     &\beta\omega(x,y,z)-\omega(\alpha(x),\alpha(y),\alpha(z))\\
     =&\theta(\alpha(x),\alpha(z))\varphi(y)-\theta(\alpha(y),\alpha(z))\varphi(x)-D(\alpha(x),\alpha(y))\varphi(z)\nonumber\\
&+\varphi([x,y,z]_{V})-[\varphi(x),\varphi(y),\varphi(z)]_{V},\nonumber
\end{align}
\begin{align}\label{Iam2}
     &\beta \nu(x,y)-\nu(\alpha(x),\alpha(y))\\
     =&\varphi(x) *_V \varphi(y)+\varphi(x *_B y)
     -\mu(\alpha(x))\varphi(y)+\mu(\alpha(y))\varphi(x),\nonumber
\end{align}
\begin{align}\label{Iam3}
     \beta(\theta(x,y)a)-\theta(\alpha(x),\alpha(y))\beta(a)=[\beta(a),\varphi(x),\varphi(y)]_{V},
\end{align}
\begin{align}\label{Iam4}
     \beta D(x,y)a-D(\alpha(x),\alpha(y))\beta(a)=[\varphi(x),\varphi(y),\beta(a)]_{V},
\end{align}
\begin{equation}\label{Iam5}
     \beta \mu(x)a-\mu(\alpha(x))\beta(a)=\beta(a) *_V \varphi(x),
\end{equation}
for all $x,y,z\in B$ and $a\in V$.
\end{thm}

\begin{proof}
Suppose that the pair $(\alpha,\beta)\in \mathrm{Aut}(B)\times \mathrm{Aut}(V)$ is inducible, meaning there exists an automorphism $\gamma\in \mathrm{Aut}_{V}(\hat{B})$ such
that $\gamma i=i\beta$ and $p\gamma =\alpha p$. Given that $s$ is a section of $p$, for every $x\in B$,
$$p(s\alpha-\gamma s)(x)=\alpha(x)-\alpha(x)=0,$$
which implies that $(s\alpha-\gamma s)(x)\in \mathrm{ker}p=V$.
So we can define a linear map $\varphi:B\longrightarrow
V$ by
$$\varphi(x)=(s\alpha-\gamma s)(x),~~\forall~x\in B.$$
Utilizing Eqs.~(\ref{C2}) and (\ref{C3}), for elements $x,y\in B$ and $a\in V$, we derive the following results:
\begin{eqnarray*}&&
      \beta(D(x,y)a)-D(\alpha(x),\alpha(y))\beta(a)
      \\&=&\beta[s(x),s(y),a]_{\hat{B}}-[s\alpha(x),s\alpha(y),\beta(a)]_{\hat{B}}
      \\&=&[\beta s(x),\beta s(y),\beta(a)]_{\hat{B}}-[s\alpha(x),s\alpha(y),\beta(a)]_{\hat{B}}
      \\&=&[\beta s(x)-s\alpha(x),\beta s(y),\beta(a)]_{\hat{B}}+[s\alpha(x),\beta s(y),\beta(a)]_{\hat{B}}-[s\alpha(x),s\alpha(y),\beta(a)]_{\hat{B}}
      \\&=&[-\varphi(x),\beta s(y),\beta(a)]_{\hat{B}}+[s\alpha(x),-\varphi(y),\beta(a)]_{\hat{B}}
\\&=&-[\varphi(x),\beta s(y)-s\alpha(y),\beta(a)]_{\hat{B}}-[\varphi(x),s\alpha(y),\beta(a)]_{\hat{B}}- [s\alpha(x),\varphi(y),\beta(a)]_{\hat{B}}
\\&=&[\varphi(x),\varphi(y),\beta(a)]_{\hat{B}}-[\varphi(x),s\alpha(y),\beta(a)]_{\hat{B}}+[\beta(a),\varphi(y),s\alpha(x)]_{\hat{B}}
\\&=&[\varphi(x),\varphi(y),\beta(a)]_{V}-\rho(\alpha(x))(\varphi(y),\beta(a))+\rho(\alpha(y))(\varphi(x),\beta(a)),
\end{eqnarray*}
This demonstrates that Eq.~(\ref{Iam4}) is satisfied. By employing the same approach, we can establish that Eqs.~(\ref{Iam1})-(\ref{Iam3}) also hold.

Given that $s$ is a section of $p$, any element $\hat{w}\in \hat{B}$ can be uniquely decomposed as $\hat{w}=a+s(x)$, where $a\in V$ and $x\in B$. Suppose there exists a linear map $\varphi:B\longrightarrow V$ hat satisfies Eqs. (\ref{Iam1})-(\ref{Iam5}). Under these conditions, we define a linear map $\gamma:\hat{B}\longrightarrow
\hat{B}$ by
$$\gamma(\hat{w})=\gamma(a+s(x))=\beta(a)-\varphi(x)+s\alpha(x).$$

First, we show that $\gamma$ is bijective. Suppose $\gamma(\hat{w})=0$. Then we have $s\alpha(x)=0$ and
$\beta(a)-\varphi(x)=0$. Since $s$ and $\alpha$ are injective, it follows that $x=0$, and consequently,  $a=0$. Therefore, $\hat{w}=a+s(x)=0$, proving that $\gamma$ is injective. For any $\hat{w}=a+s(x)\in \hat{B}$,
$$\gamma(\beta^{-1}(a)+\beta^{-1}\varphi\alpha^{-1}(x)+s\alpha^{-1}(x))=a+s(x)=\hat{w},$$
which yields that $\gamma$ is surjective. In all, $\gamma$ is
bijective.

To further substantiate our argument, we next demonstrate that $\gamma$ serves as a homomorphism for the Bol algebra $\hat{B}$. Specifically, for any elements $\hat{w}_i=a_i+s(x_i)\in \hat{B}$ where $i=1,2,3$,
\begin{eqnarray*}&&
      [\gamma(\hat{w}_1),\gamma(\hat{w}_2),\gamma(\hat{w}_3)]_{\hat{B}}
\\&=&[\beta(a_1),\beta(a_2),\beta(a_3)]_{\hat{B}}-[\beta(a_1),\beta(a_2),\varphi(x_3)]_{\hat{B}}-[\beta(a_1),\varphi(x_2),\beta(a_3)]_{\hat{B}}
\\&&+[\beta(a_1),\varphi(x_2),\varphi(x_3)]_{\hat{B}}+\theta(\alpha(x_2),\alpha(x_3))\beta(a_1)-[\varphi(x_1),\beta(a_2),\beta(a_3)]_{\hat{B}}
\\&&+[\varphi(x_1),\beta(a_2),\varphi(x_3)]_{\hat{B}}+[\varphi(x_1),\varphi(x_2),\beta(a_3)]_{\hat{B}}-[\varphi(x_1),\varphi(x_2),\varphi(x_3)]_{\hat{B}}
\\&&-\theta(\alpha(x_2),\alpha(x_3))\varphi(x_1)-\theta(\alpha(x_1),\alpha(x_3))\beta(a_2)+\theta(\alpha(x_1),\alpha(x_3))\varphi(x_2)
\\&&+D(\alpha(x_1),\alpha(x_2))\beta(a_3)-D(\alpha(x_1),\alpha(x_2))\varphi(x_3)+\omega(\alpha(x_1),\alpha(x_2),\alpha(x_3))\\
&&+s\alpha[x_1,x_2,x_3]_{B}
\end{eqnarray*}
and
\begin{eqnarray*}&&
      \gamma([\hat{w}_1,\hat{w}_2,\hat{w}_3]_{\hat{B}})
\\&=&[\beta(a_1),\beta(a_2),\beta(a_3)]_{\hat{B}} +\theta(x_2,x_3)a_1-\theta(x_1,x_3)a_2+D(x_1,x_2)a_3
\\&&+\beta([s(x_1), s(x_2),s(x_3)]_{\hat{B}})+\beta \omega(x_1,x_2,x_3)+s\alpha [x_1,x_2,x_3]_{B}-\varphi([x_1,x_2,x_3]_{B}).
\end{eqnarray*}
Thanks to Eqs.~~(\ref{Iam1})-(\ref{Iam5}), we have
$$ \gamma([\hat{w}_1,\hat{w}_2,\hat{w}_3]_{\hat{B}})=[\gamma(\hat{w}_1),\gamma(\hat{w}_2),\gamma(\hat{w}_3)]_{\hat{B}}.$$
Similarly, it can be shown that $ \gamma((\hat{w}_1)*_{\hat{B}}(\hat{w}_2))=\gamma(\hat{w}_1) *_{\hat{B}} \gamma(\hat{w}_2).$ Consequently, $\gamma$ qualifies as an element of $\mathrm{Aut}_{V}(\hat{B})$. This concludes the proof.
\end{proof}

For all $(\alpha,\beta)\in \mathrm{Aut}(B)\times \mathrm{Aut}(V)$, we define the maps  $\nu_{(\alpha,\beta)}:B\otimes
B\longrightarrow V,~\omega_{(\alpha,\beta)}:B\otimes
B\otimes B\longrightarrow V,~\mu_{(\alpha,\beta)}:B\longrightarrow \mathfrak{gl}(V),~\theta_{(\alpha,\beta)},
D_{(\alpha,\beta)}:B\wedge B\longrightarrow \mathfrak{gl}(V)$ based on the non-abelian (2,3)-cocycle $(\nu,\omega,\mu,\theta,D)$ induced by the section $s$ of $p$ in the non-abelian extension $\mathcal{E}:0\longrightarrow V\stackrel{i}{\longrightarrow}
\hat{B}\stackrel{p}{\longrightarrow}B\longrightarrow0$. Defined as follows:
 \begin{align}\label{Inc1}&\omega_{(\alpha,\beta)}(x,y,z)=\beta\omega(\alpha^{-1}(x),\alpha^{-1}(y),\alpha^{-1}(z)),\\~~
 &\nu_{(\alpha,\beta)}(x,y)=\beta\nu(\alpha^{-1}(x),\alpha^{-1}(y)),\nonumber\end{align}
 \begin{align}\label{Inc2}&\theta_{(\alpha,\beta)}(x,y)a=\beta(\theta(\alpha^{-1}(x),\alpha^{-1}(y))\beta^{-1}(a))
,\\~~&D_{(\alpha,\beta)}(x,y)a=\beta D(\alpha^{-1}(x),\alpha^{-1}(y))\beta^{-1}(a),\nonumber\end{align}
\begin{equation}\label{Inc3}\mu_{(\alpha,\beta)}(x)a=\beta\mu(\alpha^{-1}(x))\beta^{-1}(a),\end{equation}
for all $x,y,z\in B,a,b\in V.$

\begin{pro}Utilizing the aforementioned notations, the tuple $(\nu,\omega,\mu,\theta,D)_{(\alpha,\beta)}$ constitutes a non-abelian (2,3)-cocycle.
\end{pro}

\begin{proof}
By Eqs.~(\ref{B2}) and (\ref{Inc2}), for all $x_1,x_2,y_1,y_2,y_3\in B$, we get
\begin{eqnarray*}&&D_{(\alpha,\beta)}(x,y)\theta_{(\alpha,\beta)}(z,w)a-\theta_{(\alpha,\beta)}(z,w)D_{(\alpha,\beta)}(x,y)a
\\&=&D_{(\alpha,\beta)}(x,y)\beta(\theta(\alpha^{-1}(z),\alpha^{-1}(w))\beta^{-1}(a))-\theta_{(\alpha,\beta)}(z,w)\beta D(\alpha^{-1}(x),\alpha^{-1}(y))\beta^{-1}(a)
\\&=&\beta(\theta(\alpha^{-1}([x,y,z]_B),\alpha^{-1}(w))\beta^{-1}(a))+\beta(\theta(\alpha^{-1}(z),\alpha^{-1}([x,y,w]_B))\beta^{-1}(a))
\\&=&\theta_{(\alpha,\beta)}([x,y,z]_B,w)a+\theta_{(\alpha,\beta)}(z,[x,y,w]_B)a
 \end{eqnarray*}
 This indicates that Eq.~(\ref{B2}) is valid. By the same token, we can verify that Eqs.~(\ref{B12})-(\ref{B4}) also hold true. Thus, the proof is concluded.
\end{proof}

\begin{thm} \label{Eth1} Consider a non-abelian extension of the Bol algebra $B$ by $V$, given by $0\longrightarrow V\stackrel{i}{\longrightarrow}
\hat{B}\stackrel{p}{\longrightarrow}B\longrightarrow0$, which includes a section $s$ of $p$. Let $(\nu,\omega,\mu,\theta,D)$ denote the non-abelian (2,3)-cocycle induced by $s$. For a pair\((\alpha,\beta)\in\mathrm{Aut}(B)\times\mathrm{Aut}(V)\)to be inducible,it is necessary and sufficient that the non-abelian (2,3)-cocycles $(\nu,\omega,\mu,\theta,D)$ and
$(\nu,\omega,\mu,\theta,D)_{(\alpha,\beta)}$ are equivalent.
\end{thm}

\begin{proof}
Assume that $(\alpha,\beta)\in \mathrm{Aut}(B)\times \mathrm{Aut}(V)$ is inducible. According to Theorem ~\ref{EC}, there exists a linear map $\varphi:B\longrightarrow V$ that satisfies Eqs.~~(\ref{Iam1})-(\ref{Iam5}). For any $x,y\in B$ and $a\in V$, there exist $x_0,y_0\in B$ and $a_0\in V$ such that $x=\alpha(x_0),y=\alpha(y_0),a=\beta(a_0)$. Utilizing Eqs.~(\ref{Iam4}) and (\ref{Inc3}), we derive:
\begin{eqnarray*}&&
D_{(\alpha,\beta)}(x,y)a-D(x,y)a\\&=&\beta(D(\alpha^{-1}(x),\alpha^{-1}(y))\beta^{-1}(a))
-D(x,y)a
\\&=& \beta(D(x_0,y_0)a_0)-D(\alpha(x_0),\alpha(y_0))\beta(a_0)
\\&=&[\varphi(x_0),\varphi(y_0),\beta(a_0)]_{V}
\\&=&[\varphi\alpha^{-1}(x),\varphi\alpha^{-1}(y),a]_{V},
 \end{eqnarray*}

which demonstrates that  Eq.~(\ref{E5}) is satisfied. Similarly, Eqs.~(\ref{E1})-(\ref{E4}) can be shown to hold. Therefore, the non-abelian (2,3)-cocycles$(\nu,\omega,\mu,\theta,D)$ and
$(\nu,\omega,\mu,\theta,D)_{(\alpha,\beta)}$ are equivalent through the linear map $\varphi\alpha^{-1}:B\longrightarrow V$. The reverse implication can be established in a similar manner.

\end{proof}

\section{Wells exact sequences}
In this section, we consider the Wells map associated with non-abelian extensions of Bol algebras.
Let
\begin{equation*}\mathcal{E}:0\longrightarrow V\stackrel{i}{\longrightarrow}
\hat{B}\stackrel{p}{\longrightarrow}B\longrightarrow0
\end{equation*}
be a non-abelian extension of $B$ by $V$ with a section $s$ of $p$ and
$(\nu,\omega,\mu,\theta,D)$ be the corresponding non-abelian (2,3)-cocycle
induced by $s$.
Define a map $\mathcal{W}:\mathrm{Aut}(B)\times \mathrm{Aut}(V)\longrightarrow H^{(2,3)}_{nab}(B,V)$ by
\begin{equation}\label{W1}
	\mathcal{W}(\alpha,\beta)=[(\nu,\omega,\mu,\theta,D)_{(\alpha,\beta)}
-(\nu,\omega,\mu,\theta,D)].
\end{equation}
The map $\mathcal {W}$ is called the Wells map associated with $\mathcal{E}$.
 \begin{pro} \label{Wm1}
	The Wells map $\mathcal{W}$ does not depend on the choice of sections.
 \end{pro}

\begin{proof}
Given any $x,y,z\in B$, there exist elements $x_0,y_0,z_0\in B$ such that $x=\alpha(x_0),y=\alpha(y_0)$ and $z=\beta(z_0)$. Suppose that $(\nu',\omega',\mu',\theta',D{'}) $ is another non-abelian (2,3)-cocycle associated with the non-abelian extension $\mathcal{E}$. We know that this cocycle is equivalent to $(\nu,\omega,\mu,\theta,D)$ through a linear map $\varphi:B\to V$. According to Eqs.~(\ref{E5}) and (\ref{Inc1})-(\ref{Inc2}), we define $\psi=\beta\varphi\alpha^{-1}$, which leads to
	\begin{align*}
		&D{'}_{(\alpha,\beta)}(x,y)z-D_{(\alpha,\beta)}(x,y)z
\\=&\beta D{'}(\alpha^{-1}(x),\alpha^{-1}(y))\beta^{-1}(z)-\beta D(\alpha^{-1}(x),\alpha^{-1}(y))\beta^{-1}(z)\\
		=&\beta D{'}(x_0,y_0)z_0-\beta D(x_0,y_0)z_0\\
		=& \beta ([\varphi(x_0),\varphi(y_0),z_0]_V)
\\=& \beta ([\beta^{-1}\psi(x),\beta^{-1}\psi(y),\alpha^{-1}(z)]_V)
\\ =&[\psi(x),\psi(y),z]_{V}.\end{align*}
By the same token,
\begin{align*}
\omega{'}_{(\alpha,\beta)}(x,y,a)-\omega_{(\alpha,\beta)}(x,y,a)&= \theta_{(\alpha,\beta)}(x,z)\psi(a)-D_{(\alpha,\beta)}(x,y)\psi(a)-\theta_{(\alpha,\beta)}(y, a)\psi(x)\\
 &\quad-[\psi(x),\psi(y),\psi(a)]_V+\psi([x, y,z]_B),\\
\nu{'}_{(\alpha,\beta)}(x,y)-\nu_{(\alpha,\beta)}(x,y)&=\psi(x) *_V \psi(y)+\psi(x *_B y) -\mu_{(\alpha,\beta)}(x)\psi(y)\\
 &\quad+\mu_{(\alpha,\beta)}(y)\psi(x),\\
\theta{'}_{(\alpha,\beta)}(x,y)a-\theta_{(\alpha,\beta)}(x,y)a&=[a,\psi(x),\psi(y)]_{V}, \\
 \mu{'}_{(\alpha,\beta)}(x)a-\mu_{(\alpha,\beta)}(x)a&=a *_V \psi(x).
 \end{align*}
So, $(\nu',\omega',\mu',\theta',D{'})_{(\alpha,\beta)}$ and $(\nu,\omega,\mu,\theta,D)_{(\alpha,\beta)}$ are equivalent non-abelian (2,3)-cocycles via the linear map $\psi=\beta\varphi\alpha^{-1}$.
Combining the results of the last two paragraphs, we know that
 \begin{align*}(\nu',\omega',\mu',\theta',D{'})_{(\alpha,\beta)}-(\nu',\omega',\mu',\theta',D{'})\end{align*}
 and \begin{align*} (\nu,\omega,\mu,\theta,D)_{(\alpha,\beta)}-(\nu,\omega,\mu,\theta,D)\end{align*}
 are equivalent via the linear map $\beta\varphi\alpha^{-1}-\varphi$.
\end{proof}

  \begin{pro} \label{Wm2} Let
$\mathcal{E}:0\longrightarrow V\stackrel{i}{\longrightarrow}
\hat{B}\stackrel{p}{\longrightarrow}B\longrightarrow0$
be a non-abelian extension of $B$ by $V$  with a section $s$ of $p$. Define a map
\begin{equation}\label{W3}{\mathcal{K}}:\mathrm{Aut}_{V}(\hat{B})\longrightarrow \mathrm{Aut}(B)\times \mathrm{Aut}(B),~~{\mathcal{K}}(\gamma)=(p\gamma s,\gamma|_{V}),~\forall~\gamma\in \mathrm{Aut}_{V}(\hat{B}).\end{equation}
 Then ${\mathcal{K}}$ is a homomorphism of groups.
 \end{pro}

\begin{proof}
This is by direct computations.
\end{proof}
Consider the set $\mathrm{Aut}_{V}^{B}(\hat{B})=\{\gamma \in \mathrm{Aut}(\hat{B})| \mathcal{K}(\gamma)=(id_{B},id_{V}) \}$. Let $\mathcal{E}:0\longrightarrow V\stackrel{i}{\longrightarrow}
\hat{B}\stackrel{p}{\longrightarrow}B\longrightarrow0$ be a non-abelian extension of $B$ by $V$ with a section $s$ of $p$. Assume that $(\nu,\omega,\mu,\theta,D)$ is a non-abelian (2,3)-cocycle induced by the section $s$.

Denote
\begin{align}
		Z_{nab}^{1}(B,V)=&\left\{\varphi:B\rightarrow V\left|\begin{aligned}&a *_V \varphi(x)=[a,\varphi(x),b]_{V}=[b,a,\varphi(x)]_{V}=0,
    \\&\mu(x)\varphi(y)-\mu(y)\varphi(x)= \varphi(x *_B y)+\varphi(x) *_V \varphi(y),~
    \\&\theta(x,z)\varphi(y)-\theta(y,z)\varphi(x)-D(x,y)\varphi(z)\\&=[\varphi(x),\varphi(y),\varphi(z)]_{V}-\varphi([x,y,z]_{B}),
     \end{aligned}\right.\right\}.\label{W5}
	\end{align}
One can readily verify that $Z_{nab}^{1}(B,V)$ forms an abelian group. This group is referred to as the non-abelian 1-cocycle on $B$ with values in $V$..

 \begin{pro} \label{Wm3} Utilizing the aforementioned notations, we arrive at the following conclusions:
\begin{enumerate}[label=$(\roman*)$,leftmargin=15pt]

\item The linear map \begin{equation}\label{W6}{\mathcal{S}}(\gamma)(x)=\varphi_{\gamma}(x)=s(x)-\gamma s(x),~\forall~~\gamma\in \mathrm{Aut}_{V}^{B}(\hat{B}),~x \in B\end{equation} acts as a group homomorphism.
\item ${\mathcal{S}}$ is an isomorphism, implying that $\mathrm{Aut}_{V}^{B}(\hat{B})\simeq Z_{nab}^{1}(B,V)$.
\end{enumerate}
\end{pro}

\begin{proof}
By Eqs.~(\ref{C1})-(\ref{C3}), (\ref{W5}) and (\ref{W6}), for all $x,y,z\in B$, we have,
\begin{align*}
&\{\varphi_{\gamma}(x),\varphi_{\gamma}(y),\varphi_{\gamma}(z)\}_{V}+\theta(y,z)\varphi_{\gamma}(x)-\theta(x,z)\varphi_{\gamma}(y)\\
&+D(x,y)\varphi_{\gamma}(z)-\varphi_{\gamma}([x,y,z]_{B})
\\=&[s(x)-\gamma s(x),s(y)-\gamma s(y),s(z)-\gamma s(z)]_{\hat{B}}+[s(x)-\gamma s(x),s(y),s(z)]_{\hat{B}}\\
&-[s(y)-\gamma s(y),s(x),s(z)]_{\hat{B}}+[s(x),s(y),s(z)-\gamma s(z)]_{\hat{B}}\\
&+\gamma s([x,y,z]_{B})-s([x,y,z]_{B})
\\=&\gamma s([x,y,z]_{B})-[\gamma s(x),\gamma s(y),\gamma s(z)]_{\hat{B}}+[s(x),s(y),s(z)]_{\hat{B}}-s([x,y,z]_{B})
\\=&\omega(x, y,z)-\gamma \omega(x, y,z)
\\=&0.
\end{align*}

Similarly, it can be verified that $\varphi_{\gamma}$ satisfies the remaining identities in $ Z_{nab}^{1}(B,V)$, thereby establishing that, ${\mathcal{S}}$ is well-defined. For any $\gamma_1,\gamma_2\in \mathrm{Aut}_{V}^{B}(\hat{B})$ and $x\in B$, assuming ${\mathcal{S}}(\gamma_1)=\varphi_{\gamma_1}$ and ${\mathcal{S}}(\gamma_2)=\varphi_{\gamma_2}$, we can use Eqs.~ (\ref{W3}) and (\ref{W6}) to derive:
\begin{align*}{\mathcal{S}}(\gamma_1 \gamma_2)(x)&=s(x)-\gamma_1 \gamma_2s(x)
\\&=s(x)-\gamma_1(s(x)-\varphi_{\gamma_2}(x))
\\&=s(x)-\gamma_1s(x)+\gamma_{1}\varphi_{\gamma_2}(x)
\\&=\varphi_{\gamma_1}(x)+\varphi_{\gamma_2}(x),\end{align*}
which demonstrates that ${\mathcal{S}}(\gamma_1 \gamma_2)={\mathcal{S}}(\gamma_1)+{\mathcal{S}}( \gamma_2)$, confirming that ${\mathcal{S}}$ is a group homomorphism.

First, consider the inclusion map ${\mathcal{I}}:\mathrm{Aut}_{V}^{B}(\hat{B})\longrightarrow \mathrm{Aut}_{V}(\hat{B})$. By its very definition, ${\mathcal{I}}$ is a homomorphism of groups.

To show that  ${\mathcal{S}}$ is injective, take any $\gamma\in \mathrm{Aut}_{V}^{B}(\hat{B})$. We know that ${\mathcal{K}}(\gamma)=(p\gamma s,\gamma|_{V})=(id_B,id_V)$. If ${\mathcal{S}}(\gamma)=\varphi_{\gamma}=0$, then for all \(x\in B\), $\varphi_{\gamma}(x)=s(x)-\gamma s(x)=0$. This implies that  $\gamma=id_{\hat{B}}$, proving that ${\mathcal{S}}$ is injective.

To prove that  ${\mathcal{S}}$ is surjective, note that since \(s\) is a section of \(p\), every element \(\hat{x}\in\hat{B}\) can be written as \(\hat{x}=a+s(x)\) for some \(a\in V\) and \(x\in B\). Given any \(\varphi\in Z_1^{\text{nab}}(B,V)\), define a linear map \(\gamma:\hat{B}\to\hat{B}\) by \(\gamma(\hat{x})=s(x)-\varphi(x)+a\). It is clear that ${\mathcal{K}}(\gamma)=(p\gamma s,\gamma|_{V})=(id_B,id_V)$. To complete the proof, we need to show that \(\gamma\) is an automorphism of the Bol algebra \(\hat{B}\). This can be done by following the same steps as in the proof of the converse part of Theorem \ref{EC}. Thus, $\gamma\in \mathrm{Aut}_{V}^{B}(\hat{B})$, which shows that  ${\mathcal{S}}$ is surjective. Since  ${\mathcal{S}}$ is both injective and surjective, it is bijective, and we have $\mathrm{Aut}_{V}^{B}(\hat{B})\simeq Z_{nab}^{1}(B,V)$.
\end{proof}

 \begin{thm} \label{Wm4} Given a non-abelian extension $\mathcal{E}:0\longrightarrow V\stackrel{i}{\longrightarrow}
\hat{B}\stackrel{p}{\longrightarrow}B\longrightarrow0$ of $B$ by $V$ with a section $s$ of $\hat{B}$, we can construct the following exact sequence:
$$1\longrightarrow \mathrm{Aut}_{V}^{B}(\hat{B})\stackrel{\mathcal{I}}{\longrightarrow} \mathrm{Aut}_{V}(\hat{B})\stackrel{\mathcal{K}}{\longrightarrow}\mathrm{Aut}(B)\times \mathrm{Aut}(V)\stackrel{\mathcal{W}}{\longrightarrow} H^{(2,3)}_{nab}(B,V),$$
where $\mathrm{Aut}_{V}^{B}(\hat{B})$ is defined as the set $\{\gamma \in \mathrm{Aut}(\hat{B})| {\mathcal{K}}(\gamma)=(id_{B},id_{V}) \}$.\end{thm}

\begin{proof} Given the injectivity of ${\mathcal{I}}$ and the equality $\mathrm{Ker} {\mathcal{K}}=\mathrm{Im}{\mathcal{I}}$, we focus on proving $\mathrm{Ker} {\mathcal{W}}=\mathrm{Im}{\mathcal{K}}$. For any pair $(\alpha,\beta)\in \mathrm{Ker} {\mathcal{W}}$, Theorem  \ref{Eth1} ensures that $(\alpha,\beta)$ can be extended with respect to the non-abelian extension $\mathcal{E}$. This implies the existence of a $\gamma\in \mathrm{Aut}_{V}^{B}(\hat{B})$ such that $i\beta=\gamma i$ and $p\gamma=\alpha p$. Consequently, we find $\alpha=\alpha p s=p\gamma s,~\beta=\gamma|_{V}$, which confirms that $(\alpha,\beta)\in \mathrm{Im}{\mathcal{K}}$.

In the reverse direction, for any $(\alpha,\beta)\in \mathrm{Im}{\mathcal{K}}$, there is an isomorphism $\gamma\in \mathrm{Aut}_{V}(\hat{B})$ satisfying Eq.~(\ref{W3}). Considering the condition $\mathrm{Im}i=\mathrm{Ker }p$, we deduce $\alpha p=p\gamma s p=p\gamma $ and $i\beta=\gamma i$. This shows that $(\alpha,\beta)$ is inducible with respect to the non-abelian extension $\mathcal{E}$. By Theorem \ref{Eth1}, we conclude that $(\alpha,\beta)\in \mathrm{Ker} {\mathcal{W}}$. Thus, we establish that $\mathrm{Ker} {\mathcal{W}}=\mathrm{Im}{\mathcal{K}}$.
\end{proof}

Combining Proposition \ref{Wm3} and Theorem \ref{Wm4}, we have

\begin{thm} \label{Wm5} Let
$\mathcal{E}:0\longrightarrow V\stackrel{i}{\longrightarrow}
\hat{B}\stackrel{p}{\longrightarrow}B\longrightarrow0$
be a non-abelian extension of $B$ by $V$. There is an exact sequence:
$$0\longrightarrow Z_{nab}^{1}(B,V)\stackrel{\mathcal{I}\mathcal{S}^{-1}}{\longrightarrow} \mathrm{Aut}_{V}(\hat{B})\stackrel{\mathcal{K}}{\longrightarrow}\mathrm{Aut}(B)\times \mathrm{Aut}(V)\stackrel{\mathcal{W}}{\longrightarrow} H^{(2,3)}_{nab}(B,V).$$
\end{thm}

\section{Particular case}

Let's explore a specific instance based on the outcomes from the prior section. Consider an abelian extension $\mathcal{E}$ of $B$ by $V$.
 Suppose $(\nu,\omega)$ is a (2,3)-cocycle that corresponds to this extension $\mathcal{E}$. Utilizing  Eq.~(\ref{C2}) and  Eq.~(\ref{C3}), we can construct a quadruple $(V,\mu,\theta,D)$ that acts as a representation of $B$, as established in reference \cite{Issa}.

\begin{thm}[\cite{Issa}]
The structure $(B \oplus V, *_{\nu},[ \ , \ , \  ]_{\omega})$ constitutes a Bol algebra if and only if $(\nu,\omega)$ is a (2,3)-cocycle of $B$ with coefficients in the representation $(V,\mu,\theta,D)$.

The classification of abelian extensions of $B$ by $V$ is achieved through the cohomology group $H^{(2,3)}(B,V)$.
\end{thm}

Building on this foundation, we offer the following conclusion:

\begin{thm} Consider an abelian extension of $B$ by
$V$ given by $0\longrightarrow V\stackrel{i}{\longrightarrow}
\hat{B}\stackrel{p}{\longrightarrow}B\longrightarrow0$, equipped with a section $s$ of $p$. Let $(\nu,\omega)$ denote the (2,3)-cocycle induced by $s$. A pair of automorphisms $(\alpha,\beta)\in \mathrm{Aut}(B)\times \mathrm{Aut}(V)$ can be induced if and only if there exists a linear map $\varphi:B\longrightarrow V$ that fulfills the following conditions:
\begin{align}\label{AEE1}
     \beta\omega(x,y,z)-\omega(\alpha(x),\alpha(y),\alpha(z))=&
\theta(\alpha(x),\alpha(z))\varphi(y)-\theta(\alpha(y),\alpha(z))\varphi(x)
\\&\nonumber-D(\alpha(x),\alpha(y))\varphi(z)+\varphi([x,y,z]_{V}),
\end{align}
\begin{equation}\label{AEE2}
     \beta \nu(x,y)-\nu(\alpha(x),\alpha(y))=
    \mu(\alpha(y))\varphi(x)- \mu(\alpha(x))\varphi(y)+\varphi(x *_B y),
\end{equation}
\begin{equation}\label{AEE3}
     \beta(\theta(x,y)a)=\theta(\alpha(x),\alpha(y))\beta(a),~~\beta \mu(x)a=\mu(\alpha(x))\beta(a).
\end{equation}
\begin{equation}\label{AEE4}
\beta D(x,y)a=D(\alpha(x),\alpha(y))\beta(a).
\end{equation}
\end{thm}

\begin{proof}
This conclusion follows immediately from Theorem \ref{EC}.
\end{proof}

For any $(\alpha,\beta)\in \mathrm{Aut}(B
)\times \mathrm{Aut}(V)$, the action of $(\nu,\omega)_{(\alpha,\beta)}$ does not necessarily form a(2,3)-cocycle. To be precise, $(\nu,\omega)_{(\alpha,\beta)}$ is a (2,3)-cocycle only if Eq.~(\ref{AEE3}) holds. In light of this, it is particularly natural to introduce the space of compatible pairs of automorphisms.
\begin{align*}
		C_{(B,V)}=&\left\{(\alpha,\beta)\in \mathrm{Aut}(B
)\times \mathrm{Aut}(V)\left|\begin{aligned}&\beta(\theta(x,y)a)=\theta(\alpha(x),\alpha(y))\beta(a),
\\&\beta \mu(x)a=\mu(\alpha(x))\beta(a),~\forall~x,y\in {B},a\in {V}
     \end{aligned}\right.\right\}.
	\end{align*}

Similar to Theorem \ref{Eth1}, we obtain the following result.

\begin{thm}\label{Wm6}
Let $\mathcal{E}:0\longrightarrow V\stackrel{i}{\longrightarrow}
\hat{B}\stackrel{p}{\longrightarrow}B\longrightarrow0$ with a section $s$ of $p$, and let $(\nu,\omega)$ be a (2,3)-cocycle associated with $\mathcal{E}$. A pair $(\alpha,\beta)\in C_{(B,V)}$ is extensible with respect to the abelian extension $\mathcal{E}$ if and only if $(\nu,\omega)$ and $(\nu,\omega)_{(\alpha,\beta)}$ are in the same cohomological class.
\end{thm}

By Theorem \ref{Wm5} and Theorem \ref{Wm6}, we obtain the following result:

\begin{thm}
Consider an abelian extension of $B$ by $V$ given by $0\longrightarrow V\stackrel{i}{\longrightarrow}
\hat{B}\stackrel{p}{\longrightarrow}B\longrightarrow0$. Then we have the following Wells exact sequence:
$$0\longrightarrow H^{1}(B,V)\stackrel{\mathcal{I}\mathcal{S}^{-1}}{\longrightarrow} \mathrm{Aut}_{V}(\hat{B})\stackrel{\mathcal{K}}{\longrightarrow}C_{(B,V)}\stackrel{\mathcal{W}}{\longrightarrow} H^{(2,3)}(B,V).$$
\end{thm}


\end{document}